\documentclass[journal,twoside,web]{ieeecolor}

\usepackage{amsthm}
\usepackage{times}
\usepackage{cite}
\usepackage{amssymb}
\usepackage{mathtools}
\usepackage{graphicx}
\usepackage{graphicx,xcolor}
\usepackage{hyperref}
\usepackage{algpseudocode}
\usepackage{algorithm}
\usepackage{subfigure}
\usepackage{booktabs}
\usepackage[font=footnotesize]{caption}
\usepackage{empheq}
\usepackage{tcolorbox}
\let\labelindent\relax
\usepackage{tikz,calc,bbding}
\usetikzlibrary{shapes, arrows, decorations, decorations.pathreplacing, decorations.pathmorphing, decorations.markings, fit, matrix}
\usepackage{arydshln}
\usepackage{fancyhdr}
\usepackage{ragged2e}

\graphicspath{{epsfiles/}}

\definecolor{subsectioncolor}{RGB}{1,1,0}

\newtheorem{theorem}{Theorem}[section]
\newtheorem{proposition}[theorem]{Proposition}
\newtheorem{lemma}[theorem]{Lemma}

\newtheorem{definition}[theorem]{Definition}
\newtheorem{assumption}[theorem]{Assumption}
\theoremstyle{definition}
\newtheorem{remark}[theorem]{Remark}
\newtheorem{example}[theorem]{Example}

\newcommand{\longthmtitle}[1]{\mbox{}{\textit{(#1):}}}
\makeatletter
\newcommand{\Vast}{\bBigg@{4.5}}
\makeatother

\usepackage{enumitem}
\renewcommand{\theenumi}{(\roman{enumi})}

\newcommand{\rowrank}{\textnormal{rowrk }}

\newcommand{\oprocendsymbol}{\hbox{$\square$}}
\newcommand{\oprocend}{\relax\ifmmode\else\unskip\hfill\fi\oprocendsymbol}

\renewcommand{\baselinestretch}{.985}
\allowdisplaybreaks

\newcommand{\algorithmicgiven}{\textbf{Given}}
\renewcommand{\algorithmicrequire}{\textbf{Input:}}
\renewcommand{\algorithmicensure}{\textbf{Output:}}

\makeatletter
\def\footnoterule{\kern-3\p@
  \hrule \@width 2in \kern 2.6\p@} 
\makeatother

\def\BibTeX{{\rm B\kern-.05em{\sc i\kern-.025em b}\kern-.08em
    T\kern-.1667em\lower.7ex\hbox{E}\kern-.125emX}}

\title{\bf \LARGE Data-Driven Optimal Control of Bilinear Systems}
\author{Zhenyi Yuan and Jorge Cort\'es
\thanks{The authors are with the Department of Mechanical and Aerospace Engineering, UC San
    Diego, {\tt \{z7yuan,cortes\}@ucsd.edu}. Work  supported by AFOSR
     FA9550-19-1-0235 and NSF 2044900.}}

\fancypagestyle{FirstPage}{\vspace{-1.5ex}
\lfoot{\justifying \footnotesize \copyright 2023 IEEE.  Personal use of this material is permitted.  Permission from IEEE must be obtained for all other uses, in any current or future media, including reprinting/republishing this material for advertising or promotional purposes, creating new collective works, for resale or redistribution to servers or lists, or reuse of any copyrighted component of this work in other works.} 
}

\begin{document}

\pagestyle{empty}
\maketitle
\thispagestyle{FirstPage}

\begin{abstract}
  This paper develops a method to learn optimal controls from data for
  bilinear systems without a priori knowledge of the dynamics. Given
  an unknown bilinear system, we characterize when the available data
  is sufficiently informative to solve the optimal control
  problem. This characterization leads us to propose an online control
  experiment design procedure that guarantees that any input/state
  trajectory can be represented as a linear combination of collected
  input/state data matrices. Leveraging this representation, we
  transform the original optimal control problem into an equivalent
  data-based optimization problem with bilinear constraints. We solve
  the latter by iteratively employing a convex-concave procedure to
  find a locally optimal control sequence. Simulations show that the
  performance of the proposed data-based approach is comparable with
  model-based methods.
\end{abstract}


\vspace{-2ex}
\section{Introduction}\label{sec:intro}

\IEEEPARstart{T}{he} widespread availability of data, together with increasing
computational capabilities to store, process, and manipulate it, has
boosted the research activity in learning, modeling, and control of
dynamical phenomena across science and engineering. Data-driven
control has emerged as an appealing way of leveraging this data surge
by employing solid theoretical principles to design controllers that
do not require explicit a priori knowledge of the plant to be
controlled. This paper contributes to this body of work by studying
the data-driven synthesis of optimal control laws for bilinear
systems.

\emph{Literature Review.} Data-driven control approaches include
indirect and direct methods~\cite{ZH-ZW:13}. Indirect methods identify
system models from data prior to proceeding to the synthesis of
model-based controllers, while direct approaches bypass the
intermediate modeling step and construct controllers directly from
data. A diverse range of factors, including the complexity of the
plant, the cost and practicality of performing system identification,
and the amount and quality of the available data, play a key role in
the suitability and performance of each of these approaches, see
e.g.,~\cite{VK-FP:21,FD-JC-IM:21}. The direct data-driven approach has
been particularly fruitful for linear systems, where tools from
behavioral theory~\cite{JCW-PR-IM-BLMDM:05} have allowed to express
the system trajectories in terms of sufficiently-rich data. This has
resulted in the synthesis of feedback stabilizing
controllers~\cite{CDP-PT:19,HJVW-JE-HLT-MKC:20}, optimal control
laws~\cite{GRGBDS-SA-CL-LC:18,GB-VK-FP:19,GB-FP:20}, predictive
controllers~\cite{JC-JL-FD:19,AA-JC:21-csl}, network
controllers~\cite{AA-JC:21-csl,GB-DSB-FP:21}, control experiment
design~\cite{HJVW:21}, optimization-based
controllers~\cite{GB-MV-JC-EDA:21-tac}, and extensions to various
types of nonlinear systems~\cite{AB-CDP-PT:20,CDP-PT:21,JB-FA:20},
including flat \cite{MA-JB-VGL-FA-MAM:21},
  second-order~\cite{JGRE-JS:20}, and linear time-varying
  systems~\cite{CV-RT-SH-AK:21}. Here we focus on direct
data-driven control of bilinear systems as a building block for future
work on more complex nonlinear systems. These systems are often viewed
as the bridge between linear and nonlinear systems due to their
special properties~\cite{CB-GD-GK:74}.
Moreover, \cite{DG-DAP:21} shows that control-affine nonlinear systems
can be exactly bilinearized. The recent work~\cite{AB-CDP-PT:20}
proposes a local stabilizing data-driven controller design for
bilinear systems. Here, we focus on the synthesis of optimal
controllers.  Model-based approaches to optimal control of bilinear
systems include~\cite{ZA-ZG:94,SW-JL:17,SW-JL:18}, which treat them as
time-varying linear systems and solve the optimization problem by
applying iteratively the Pontryagin’s maximum principle, and
\cite{YZ-JC:17-tcns}, which gives a lower bound on the minimum control
energy required to steer the bilinear system using the reachability
Gramian.

\emph{Statement of Contributions.} We consider discrete-time bilinear
control systems and study the point-to-point optimal control problem
over a finite time horizon. We assume the system matrices are unknown
and seek to learn the optimal control from input/state data. We
introduce the notion of $T$-persistently exciting data to characterize
when it is sufficiently informative for reconstructing the optimal
control over the time horizon $T$. Under this hypothesis, we show that
any input/state trajectory can be represented as a linear combination
of the collected input/state data. Owing to the nonlinear nature of
bilinear systems, the problem of ensuring that data is
$T$-persistently exciting requires us to introduce an online control
experiment design. We show our design is guaranteed to yield
$T$-persistently exciting data in a finite number of steps. Building
on this, we pose the optimal control synthesis problem as a data-based
optimization with bilinear constraints. We show that a local solution
to this nonconvex problem can be found by iteratively solving the
convexified problems that result from applying a convex-concave
approximation procedure.
Simulations show similar performance between the proposed data-based
approach and model-based methods.

\section{Problem Formulation}\label{sec:2}

Consider\footnote{We denote by
  $\mathbb{R}$, $\mathbb{Z}_{\ge 0}$, and $\mathbb{Z}_{> 0}$ the sets
  of real, non-negative integer, and positive integer numbers,
  resp. Let $I$, and $\mathbf{0}$ and $\mathbf{1}$ denote the identity
  matrix, and zero and all-ones vector/matrix, resp. Given
  $f: \mathbb{Z}_{\ge 0} \to \mathbb{R}^d$ and
  $i,j \in \mathbb{Z}_{\ge 0}$, $i \leq j$, $f_{[i,j]}$ is the
  restriction of $f$ to $[i, j]$ in vector form, i.e.,
  $f_{[i,j]} = [ f(i)^\top \ f(i+1)^\top \ \cdots \ f(j)^\top]^\top$,
  and $f_{\{i,j\}}$ the sequence $\{f(i),\dots,f(j)\}$. For
  $\mathbf{X} = [\mathbf{x}_1^\top \ \mathbf{x}_2^\top \ \cdots \
  \mathbf{x}_j^\top]^\top \in \mathbb{R}^{ij}$ with
  $\mathbf{x}_1,\cdots,\mathbf{x}_j \in \mathbb{R}^i$,
  $\mathcal{H}_k(\mathbf{X})$ denotes the Hankel matrix of depth
  $k \in \mathbb{Z}_{> 0}$, with $k \leq j$,
\begin{align*}
    \mathcal{H}_k(\mathbf{X}) := \left[ \begin{matrix} \mathbf{x}_1 & \mathbf{x}_2 & \cdots & \mathbf{x}_{j-k+1} \\
        \mathbf{x}_2 & \mathbf{x}_3 & \cdots & \mathbf{x}_{j-k+2} \\
    \vdots & \vdots & \ddots & \vdots \\ 
    \mathbf{x}_k & \mathbf{x}_{k+1} & \cdots & \mathbf{x}_j\\ \end{matrix} \right] \in \mathbb{R}^{ik \times (j-k+1)}.
\end{align*}
Given matrices $\mathbf{Y}$ and $\mathbf{Z}$,
$[\mathbf{Y} \ \mathbf{Z}]$ and $[\mathbf{Y}; \mathbf{Z}]$ denote
their row- and column-concatenations, resp. We use $\mathbf{Z}^\dag$
and $\operatorname{Im} \mathbf{Z}$ to represent the pseudo-inverse and
image space of $\mathbf{Z}$, resp. Finally, $\otimes$ denotes the
Kronecker product, while $\left\| \cdot \right\|$ represents the
Euclidean norm.}  the discrete-time bilinear control system
\begin{align}\label{system}
  \mathbf{x}(t+1) = \mathbf{A}\mathbf{x}(t) + \mathbf{B}\mathbf{u}(t) + \Big[ \sum_{j=1}^{n} \mathbf{x}_j(t) \mathbf{N}_j \Big] \mathbf{u}(t),
\end{align}
where $\mathbf{x}(t) \in \mathbb{R}^n$ and $\mathbf{u}(t) \in \mathbb{R}^m$ are the system state and input, respectively, and $\mathbf{A} \in \mathbb{R}^{n \times n}$, $\mathbf{B} \in \mathbb{R}^{n \times m}$ and $\mathbf{N}_j \in \mathbb{R}^{n \times m}, j = 1,\dots,n$ are system matrices. Denoting $\mathbf{N} = [\mathbf{N}_1 \ \mathbf{N}_2 \ \cdots \ \mathbf{N}_n] \in \mathbb{R}^{n \times mn}$, the dynamics~\eqref{system} is
\begin{align}\label{system2}
    \mathbf{x}(t+1) = \mathbf{A}\mathbf{x}(t) + \mathbf{B}\mathbf{u}(t) + \mathbf{N} (\mathbf{x}(t)\otimes \mathbf{u}(t)).
\end{align}
We make the following assumption.

\begin{assumption}\label{ass:controllable-pair}
The pair $(\mathbf{A}, [\mathbf{B} \ \mathbf{N}])$ is controllable.
\end{assumption}

Note that Assumption~\ref{ass:controllable-pair}
is weaker than asking for the bilinear system~\eqref{system2} to be controllable. Given initial $\mathbf{x}_0$ and target $\mathbf{x_f}$ states, we consider the following  (point-to-point) optimal control problem over the time horizon~$T$,
\begin{align}\label{problem}
  \min_{\mathbf{u}_{[0,T-1]}} \quad &\sum_{t=0}^{T-1} \mathbf{x}^\top(t) \mathbf{Q} \mathbf{x}(t) + \mathbf{u}^\top(t) \mathbf{R} \mathbf{u}(t) \notag
    \\
     \text{s.t.} \quad  &\mathbf{x}(t+1) = \mathbf{A}\mathbf{x}(t) + \mathbf{B}\mathbf{u}(t) + \Big[ \sum_{j=1}^{n}\mathbf{x}_j(t) \mathbf{N}_j \Big] \mathbf{u}(t), \notag
    \\
     &\mathbf{x}(0) =  \mathbf{x_0}, \ \mathbf{x}(T) = \mathbf{x_f}. \tag{P1}
\end{align}
Here, $\mathbf{Q} \in \mathbb{R}^{n \times n}$,
$\mathbf{R} \in \mathbb{R}^{m \times m}$ are positive
semi-definite. The minimum-energy control problem corresponds to
$\mathbf{Q} = \mathbf{0}$ and $\mathbf{R} = I$. This optimization is
nonconvex and its closed-form solution is not known in general. The
optimality conditions of \eqref{problem} lead to a nonlinear two-point
boundary-value problem, for which there is no analytical solution
available~\cite{YZ-JC:17-tcns}.

We address the following problem: assume the system matrices
$\mathbf{A}$, $\mathbf{B}$ and $\mathbf{N}_j, j=1,\dots,n$ are unknown
and, instead, we have access to input/state data of a control
experiment of~\eqref{system2}, that is, a control input sequence
$\mathbf{u}_{\{0,L-1\}}$ along with the corresponding state sequence
$\mathbf{x}_{\{0,L\}}$ of~\eqref{system2}.  Our objective is to
develop an algorithmic procedure to learn from the data the optimal
control sequence $\mathbf{u}^\star_{\{0,T-1\}}$ that
solves~\eqref{problem}.

\section{$T$-Persistently Exciting Data for Optimal
    Control of Bilinear
    Systems}\label{sec:data-suitability} 
In this section, we characterize when the available data is sufficient
to solve the optimal control problem and discuss a procedure to design
the control experiment. To motivate our discussion, we start by
considering the linear system
\begin{align}\label{eq:linear-system}
  \mathbf{x}(t+1) = \mathbf{A}\mathbf{x}(t) + \mathbf{B}\mathbf{u}(t) 
\end{align}
(corresponding to $\mathbf{N} = \mathbf{0}$ in~\eqref{system2}). Let
$\mathbf{x}_{\{0,L\}}$ be a state sequence generated
by~\eqref{eq:linear-system} with input sequence
$\mathbf{u}_{\{0,L-1\}}$. According to Willems' fundamental
lemma~\cite{JCW-PR-IM-BLMDM:05,HJVW-CDP-MKC-PT:20}, and assuming the
pair $(\mathbf{A}, \mathbf{B})$ is controllable, if
$\mathbf{u}_{\{0,L-1\}}$ is persistently
exciting\footnote{$f_{\{0,L-1\}}$ is persist. exciting of order
  $k$ if $\mathcal{H}_k(f_{[0,L-1]})$ is full-row rank.} of order
$n+T$, then
$\tilde{\mathcal{G}}_T(L) := [ \mathbf{x}_{[0,L-T]};
\mathcal{H}_T(\mathbf{u}_{[0,L - 1]}) ] \in \mathbb{R}^{(n+mT) \times
  (L-T+1)}$ is full-row rank. This ensures that for any input/state
trajectory ($\bar{\mathbf{u}}_{[0, T-1]},\bar{\mathbf{x}}_{[0, T-1]}$)
of length $T$ of the linear system~\eqref{eq:linear-system}, there
exist some $\tilde{\alpha} \in \mathbb{R}^{L-T+1}$ such that
\begin{align*}
         \left[\begin{array}{l}
         \bar{\mathbf{x}}_{[0, T-1]} \\ \hline
            \bar{\mathbf{u}}_{[0, T-1]}         
     \end{array}\right] = \left[\begin{array}{l}
            \mathcal{H}_{T}(\mathbf{x}_{[0, L-1]}) \\ \hline
            \mathcal{H}_{T}(\mathbf{u}_{[0, L-1]})
                \end{array}\right] \tilde{\alpha}.
\end{align*}
Now consider the original bilinear system~\eqref{system2}. If we
regard $\mathbf{x}(t) \otimes \mathbf{u}(t)$ as an independent input,
then the dynamics corresponds to a linear system with input matrix
$\left[\mathbf{B} \ \mathbf{N}\right]$ and control input
$\mathbf{v}(t) = [ \mathbf{u}(t); \mathbf{x}(t) \otimes \mathbf{u}(t)
]$. Willems' fundamental lemma applied to this linear system implies
that, under  Assumption~\ref{ass:controllable-pair}, if
$\mathbf{v}_{\{0,L-1\}}$ is persistently exciting of order $n+T$, then
$[ \mathbf{x}_{[0,L-T]}; \mathcal{H}_T(\mathbf{v}_{[0,L - 1]}) ]$ is
full-row rank, which then ensures that any input/state trajectory
($\bar{\mathbf{v}}_{[0, T-1]},\bar{\mathbf{x}}_{[0, T-1]}$) of length
$T$ of system \eqref{system2} can be represented by
\begin{align}\label{eq:linear-combination}
    \left[\begin{array}{l}
     \bar{\mathbf{x}}_{[0, T-1]} \\ \hline
    \bar{\mathbf{v}}_{[0, T-1]} 
    \end{array}\right] = \left[\begin{array}{l}
    \mathcal{H}_{T}(\mathbf{x}_{[0, L-1]}) \\ \hline
    \mathcal{H}_{T}(\mathbf{v}_{[0, L-1]})
    \end{array}\right] \alpha,
\end{align}
for some $\alpha \in \mathbb{R}^{L-T+1}$. However, as we know, the input $\mathbf{v}$ is not independent, and ensuring it is persistently exciting is not guaranteed by simply asking for $\mathbf{u}$ to be so. These observations motivate our ensuing definitions and technical treatment.

\begin{remark}\longthmtitle{Fundamental lemma for nonlinear systems}
    Our exposition above uses the fundamental lemma in the context of
    bilinear systems by interpreting the bilinear term as an input.
    Recent literature on data-driven control has pursued similar ideas
    for different classes of nonlinear systems by relying on linear
    expressions in lifted coordinates, e.g., Hammerstein and
    Wiener~\cite{JB-FA:20}, linear
    parameter-varying~\cite{CV-RT-SH-AK:21}, second-order
    Volterra~\cite{JGRE-JS:20} and flat~\cite{MA-JB-VGL-FA-MAM:21}
    systems. \hfill $\rhd$
\end{remark}

\vspace*{-3ex}
\subsection{Parametrization of state trajectories}
We next introduce the notion of $T$-persistently exciting data.

\begin{figure*}[htb]
\begin{align}\label{eq:transition}
    \mathcal{H}_T(\mathbf{x}_{[1,L]}) =  \Vast[ \underbrace{ \begin{array}{c} \mathbf{A} \\ \mathbf{A}^2 \\ \vdots \\ \mathbf{A}^T \end{array}}_{\mathcal{O}_T} \Vast| \underbrace{\begin{array}{cccc}
        \mathbf{B} & \mathbf{0} & \cdots & \mathbf{0} \\
        \mathbf{AB} & \mathbf{B} &  \cdots & \mathbf{0} \\
        \vdots & \vdots & \ddots & \vdots \\
        \mathbf{A}^{T-1}\mathbf{B} & \mathbf{A}^{T-2}\mathbf{B} & \cdots & \mathbf{B} 
    \end{array}}_{\mathcal{P}_T} \Vast| \underbrace{\begin{array}{cccc}
        \mathbf{N} & \mathbf{0} & \cdots{} & \mathbf{0} \\
        \mathbf{AN} & \mathbf{N} & \cdots & \mathbf{0}\\
        \vdots & \vdots & \ddots & \vdots \\
        \mathbf{A}^{T-1}\mathbf{N} & \mathbf{A}^{T-2}\mathbf{N} & \cdots & \mathbf{N} 
    \end{array} }_{\mathcal{Q}_T} \Vast] \mathcal{G}_T(L)
\end{align}
\vspace*{-2ex}
\end{figure*}

\begin{definition}\longthmtitle{$T$-persistently exciting
      data for optimal control of bilinear systems}\label{def}
  Let $\mathbf{x}_{\{0,L\}}$ be a state sequence generated by
  \eqref{system2} with input sequence $\mathbf{u}_{\{0,L-1\}}$.  The
  data $\mathbf{x}_{\{0,L\}}$, $\mathbf{u}_{\{0,L-1\}}$ is
  \emph{$T$-persistently exciting} if
\begin{align*}
	\mathcal{G}_T(L) \!:=\! \left[ \begin{array}{c}
     \mathcal{H}_1(\mathbf{x}_{[0,L-T]})  \\ \hline
     \mathcal{H}_T(\mathbf{u}_{[0,L-1]}) \\ \hline
     \mathcal{H}_T(\mathbf{x \!\otimes\! u}_{[0,L-1]})
\end{array}\right] \in \mathbb{R}^{(n+mT+mnT) \times (L-T+1)}.
\end{align*}
is full-row rank.
\end{definition}

This definition requires as a necessary condition that
$L \geq (mn + m + 1)T + n - 1$. We point out that $\mathcal{G}_T(L)$
being full-row rank is equivalent to
$\left[ \mathcal{H}_1(\mathbf{x}_{[0,L-T]});
  \mathcal{H}_T(\mathbf{v}_{[0,L - 1]}) \right]$ being full-row rank,
as both matrices can be obtained from each other by row
permutation. Using~\eqref{system2}, one can obtain the
relation~\eqref{eq:transition}, where
$\mathcal{O}_T \in \mathbb{R}^{nT \times n}$,
$\mathcal{P}_T \in \mathbb{R}^{nT \times mT}$,
$\mathcal{Q}_T \in \mathbb{R}^{nT \times mnT}$. If $\mathcal{G}_T(L)$
is full-row rank, it immediately follows that
\begin{align*}
  \left[ \mathcal{O}_T \ \mathcal{P}_T \ \mathcal{Q}_T \right] = \mathcal{H}_T(\mathbf{x}_{[1,L]}) \mathcal{G}_T(L)^\dag.
\end{align*}

\begin{remark}\longthmtitle{$T$-persistently exciting data for optimal control versus for identification and stabilization}\label{rmk:identifystabilize}
  When $T = 1$, we have
  $\left[ \mathcal{O}_1 \ \mathcal{P}_1 \ \mathcal{Q}_1 \right]=\left[
    \mathbf{A} \ \mathbf{B} \ \mathbf{N}
  \right]$. Therefore, $1$-persistently exciting data
    corresponds to the standard notion of persistently exciting data,
    describing data needed for system
    identification~\cite{CDP-PT:19,HJVW-JE-HLT-MKC:20}.  Moreover, if
  $\mathcal{H}_1(\mathbf{x}_{[0,L-1]})$ is full-row rank, then under
  knowledge of an upper bound on $\| \mathbf{N} \|$, one can construct
  locally stabilizing controllers directly from data,
  cf.~\cite{AB-CDP-PT:20}. Notice $\mathcal{G}_T(L)$ is full-row rank
  $\Rightarrow \mathcal{G}_1(L)$ is full-row rank
  $\Rightarrow \mathcal{H}_1(\mathbf{x}_{[0,L-1]})$ is full-row
  rank. We deduce that $T$-persistently exciting data
  comprises data needed for system identification and local
  stabilization. Note that, although $\mathcal{O}_T, \mathcal{P}_T$
  and $\mathcal{Q}_T$ can be~constructed only using
  $\mathbf{A}, \mathbf{B}$ and $\mathbf{N}$ when $\mathcal{G}_1(L)$ is
  full-row rank, this is not enough to express any input/state
  trajectory of length $T$ as a linear combination of the collected
  input/state data, and thus $\mathcal{G}_1(L)$ being full-row rank is
  not sufficient to recover optimal controls. Another observation is
  that, in the linear case ($\mathbf{N} = \mathbf{0}$), according
  to~\cite{GB-FP:20}, optimal controls over the time horizon $T$ can
  be learned if $\tilde{\mathcal{G}}_T(L)$ is full-row rank.  One can
  identify and globally stabilize linear systems directly using data
  if $\tilde{\mathcal{G}}_1(L)$ is full-row rank,
  cf.~\cite{CDP-PT:19}.  Since $\tilde{\mathcal{G}}_T(L)$ is full-row
  rank $\Rightarrow \tilde{\mathcal{G}}_1(L)$ is full-row rank, this
  reinforces the parallelism between the bilinear and linear cases
  regarding data conditions for different control problems. \hfill
  $\rhd$
\end{remark}

We have established that any input/state trajectory of~\eqref{system2}
admits a data-based representation of the
form~\eqref{eq:linear-combination}. The next result establishes when
the converse is also true, i.e., when a trajectory of the
form~\eqref{eq:linear-combination} corresponds to a trajectory
of~\eqref{system2}.

\begin{lemma}\longthmtitle{Data-based representation of input/state
    trajectory in terms of $T$-persistently exciting
    data}\label{lemma:data-based-presentation}
Let  $\mathbf{x}_{\{0,L\}}$ and $\mathbf{u}_{\{0,L-1\}}$ be a $T$-persistently exciting data set. Then
\begin{enumerate}
\item Any input/state trajectory
  ($\bar{\mathbf{u}}_{[0, T-1]},\bar{\mathbf{x}}_{[0, T]}$) of
  system~\eqref{system2} can be represented as
\begin{align*}
\left[ \begin{array}{c} \bar{\mathbf{x}}_{[0, T]} \\ \hline \bar{\mathbf{u}}_{[0, T-1]} \end{array} \right] = \left[ \begin{array}{c} \mathcal{H}_{T+1}(\mathbf{x}_{[0,L]}) \\ \hline \mathcal{H}_{T}(\mathbf{u}_{[0, L-1]}) \end{array} \right] \alpha
\end{align*}
for some  $\alpha \in \mathbb{R}^{L-T+1}$;
\item Conversely, let  $\alpha \in \mathbb{R}^{L-T+1}$ such that
\begin{align}\label{constraint-alpha}
    \bar{\mathbf{x}} \otimes \bar{\mathbf{u}}_{[0,T-1]} = \mathcal{H}_T(\mathbf{x \otimes u}_{[0,L-1]}) \alpha,
\end{align}
where $\bar{\mathbf{x}}_{[0,T-1]} = \mathcal{H}_{T}(\mathbf{x}_{[0,L-1]})  \alpha$ and $\bar{\mathbf{u}}_{[0,T-1]} = \mathcal{H}_T(\mathbf{u}_{[0,L-1]}) \alpha$. Then, 
  $\left[\mathcal{H}_{T+1}(\mathbf{x}_{[0,L]});\mathcal{H}_{T}(\mathbf{u}_{[0, L-1]}) \right] \alpha$ is an input/state trajectory of~\eqref{system2} over the time horizon~$T$.
\end{enumerate}
\end{lemma}
\begin{proof}
  For \emph{(i)}, note that any input/state trajectory
  $(\bar{\mathbf{x}}_{[0,T]},\bar{\mathbf{u}}_{[0,T-1]})$
  of~\eqref{system2} is uniquely determined by
  $\bar{\mathbf{x}}(0) \in \operatorname{Im}
  \mathcal{H}_1(\mathbf{x}_{[0,L-T]})$ and
  $\bar{\mathbf{u}}_{[0,T-1]} \in \operatorname{Im}
  \mathcal{H}_{T}(\mathbf{u}_{[0, L-1]})$. Recalling that
  $\mathcal{G}_T(L)$ is full-row rank, \emph{(i)} follows.  For
  \emph{(ii)}, let $\alpha$ satisfy \eqref{constraint-alpha} and
  consider the initial state
  $\bar{\mathbf{x}}(0) = \mathcal{H}_1(\mathbf{x}_{[0,L-T]}) \alpha$
  and input sequence
  $\bar{\mathbf{u}}_{[0,T-1]} = \mathcal{H}_{T}(\mathbf{u}_{[0, L-1]})
  \alpha$. Then,
\begin{align*}
    \bar{\mathbf{x}}_{[1,T]} &= \left[ \mathcal{O}_T \ \mathcal{P}_T \ \mathcal{Q}_T \right] \left[ \begin{array}{c}
    \bar{\mathbf{x}}(0) \\ \hline
    \mathcal{H}_T(\bar{\mathbf{x}} \otimes \bar{\mathbf{u}}_{[0,T-1]}) \\ \hline
    \bar{\mathbf{u}}_{[0,T-1]} \end{array}\right] \\
    &= \left[ \mathcal{O}_T \ \mathcal{P}_T \ \mathcal{Q}_T \right] \mathcal{G}_T(L) \alpha = \mathcal{H}_T(\mathbf{x}_{[1,L]}) \alpha,
\end{align*}
where we have employed \eqref{eq:transition}. The conclusion follows by noting $\mathcal{H}_{T+1}(\mathbf{x}_{[0,L]}) = [\mathcal{H}_1(\mathbf{x}_{[0,L-T]});\mathcal{H}_T(\mathbf{x}_{[1,L]})]$. 
\end{proof}

\subsection{Online experiment for $T$-persistence of excitation}

We discuss next how to ensure that the available data is
$T$-persistently exciting. Based on our discussion above,
$\mathbf{v}_{\{0,L-1\}}$ being persistently exciting of order $n+T$ is
enough to ensure the $T$-persistence of excitation of the data for
bilinear systems. In contrast to the linear case, where
$\mathbf{u}_{\{0,L-1\}}$ can be designed to be persistently exciting
of any order by selecting control inputs offline, the persistence of
excitation of $\mathbf{v}_{\{0,L-1\}}$ depends on both the control
input $\mathbf{u}(t)$ and the system state $\mathbf{x}(t)$. Due to the
unknown nonlinear dynamics, there is no available closed-form
expression of $\mathbf{x}(t)$ in terms of $\mathbf{u}(t)$. Hence,
selecting control inputs offline may not guarantee
$\mathbf{v}_{\{0,L-1\}}$ to be persistently exciting of order $n+T$,
which motivates an online approach to design $\mathbf{u}$.
To tackle this, we draw inspiration
  from~\cite{CDP-PT:21,HJVW:21} to propose an experiment design
  approach for bilinear systems that yields
  $T$-persistently exciting data. We start with some
useful facts.

\begin{proposition}\longthmtitle{Scaled persistently exciting input
    returns a full-row rank Hankel matrix of state
    data}\label{pro:scaled}
  Consider system~\eqref{system2} and further assume that the pair
  $(\mathbf{A},\mathbf{B})$ is controllable. Then, for any input
  sequence $\mathbf{u}_{\{0,L-1\}}$ that is persistently exciting of
  order $n+k$, there exists $\bar{\varepsilon}$ such that for all
  $\varepsilon \in (0,\bar{\varepsilon})$, the input sequence
  $\varepsilon \mathbf{u}_{\{0,L-1\}}$ with initial state
  $\mathbf{x}(0) = \mathbf{0}$ ensures
  $\mathcal{H}_k(\mathbf{x}_{[1,L]})$ is full-row rank.
\end{proposition}

We omit the proof for space reasons, but note that the result follows
by using for the higher-order case of $k \geq 1$ the same arguments
employed in~\cite{CDP-PT:21} for the case of~$k = 1$ 
  (note that the initial state $\mathbf{x}(0) = \mathbf{0}$ remains
  the same after scaling by $\varepsilon$). The next
  result generalizes~\cite[Thm.~2]{HJVW:21} to bilinear systems.

\begin{proposition}\longthmtitle{Property on the left kernel of
    $\mathcal{G}_T(t)$ when it is not full-row
    rank}\label{pro:left-kernel}
 Suppose $\mathcal{G}_T(t)$ is not full-row rank for some $t \geq T$. If
    \begin{align}\label{eq:verify}
        \left[ \begin{array}{c}
		\mathbf{x}(t-T+1) \\ \hline
		\mathbf{u}_{[t-T+1,t-1]} \\ \hline
		\mathbf{x \otimes u}_{[t-T+1,t-1]}
		\end{array}  \right] \in \operatorname{Im} \left[ \begin{array}{c}
		\mathcal{H}_1(\mathbf{x}_{[0,t-T]}) \\ \hline
		\mathcal{H}_{T-1}(\mathbf{u}_{[0,t-2]}) \\ \hline
		\mathcal{H}_{T-1}(\mathbf{x \otimes u}_{[0,t-2]})
		\end{array} \right] ,
	\end{align}
	then there must exist  $\xi \in \mathbb{R}^n$, $\eta_1, \dots, \eta_T \in \mathbb{R}^m$, and $\chi_1, \dots, \chi_T \in \mathbb{R}^{mn}$ such that the following holds
	\begin{align}\label{leftkernel}
	    \big[ \xi^\top \ \eta_1^\top \ \cdots \ \eta_T^\top \ \chi_1^\top \ \cdots \ \chi_T^\top \big] \mathcal{G}_T(t) = \mathbf{0},
	\end{align}
with at least one in $\{\eta_T,\chi_T\}$ not equal to $\mathbf{0}$.
\end{proposition}
\begin{proof}
We reason by contradiction. Suppose all vectors of the form $\left[ \xi^\top \ \eta_1^\top \ \cdots \ \eta_T^\top \ \chi_1^\top \ \cdots \ \chi_T^\top  \right]$ in the left kernel of $\mathcal{G}_T(t)$ satisfy that both $\eta_T$ and $\chi_T$ are equal to $\mathbf{0}$. Then,
\begin{align*}
	\left[ \xi^\top \ \eta_1^\top  \cdots  \eta_{T-1}^\top \ \chi_1^\top  \cdots  \chi_{T-1}^\top \right] \left[ \begin{array}{c}
\mathcal{H}_1(\mathbf{x}_{[0,t-T]}) \\ \hline
\mathcal{H}_{T-1}(\mathbf{u}_{[0,t-2]}) \\ \hline
\mathcal{H}_{T-1}(\mathbf{x \! \otimes \! u}_{[0,t-2]})
\end{array} \right] = \mathbf{0}.
\end{align*}
Combining this with~\eqref{eq:verify}, we deduce that
\begin{align*}
	\left[ \xi^\top \ \eta_1^\top \cdots \eta_{T-1}^\top \ \chi_1^\top \cdots \chi_{T-1}^\top \right] \left[\! \begin{array}{c}
		\mathcal{H}_1(\mathbf{x}_{[0,t-T+1]}) \\ \hline
		\mathcal{H}_{T-1}(\mathbf{u}_{[0,t-1]}) \\ \hline
		\mathcal{H}_{T-1}(\mathbf{x \otimes u}_{[0,t-1]})
		\end{array} \!\right]  = \mathbf{0}.
\end{align*}
Combining this with the fact that
\begin{align*}
	\left[ \begin{array}{c}
		\mathcal{H}_1(\mathbf{x}_{[1,t-T+1]}) \\ \hline
		\mathcal{H}_{T-1}(\mathbf{u}_{[1,t-1]}) \\ \hline
		\mathcal{H}_{T-1}(\mathbf{x \otimes u}_{[1,t-1]})
		\end{array} \right] = \left[ \begin{array}{ccccc}
		\mathbf{A} & \mathbf{B} & \mathbf{0} & \mathbf{N} & \mathbf{0} \\
		\mathbf{0} & \mathbf{0} & I & \mathbf{0} & \mathbf{0} \\
		\mathbf{0} & \mathbf{0} & \mathbf{0} & \mathbf{0} & I
		\end{array} \right] \mathcal{G}_T(t),
\end{align*}
we obtain
\begin{align*}
\left[ \xi^\top \mathbf{A} \ \xi^\top \mathbf{B} \ \eta_1^\top \ \cdots \ \eta_{T-1}^\top \ \xi^\top \mathbf{N} \ \chi_1^\top \ \cdots \ \chi_{T-1}^\top \right] \mathcal{G}_T(t) = \mathbf{0}.
\end{align*}
Consequently, given our hypothesis of contradiction, $\eta_{T-1}$ and $\chi_{T-1}$ must both be equal to $\mathbf{0}$. Following a similar procedure iteratively, we conclude that $\eta_{T-1}=\cdots=\eta_1=\mathbf{0}$ and $\chi_{T-1}=\cdots=\chi_1=\mathbf{0}$.
This implies that $\operatorname{Im} \mathcal{G}_T(t) = \operatorname{Im} \mathcal{H}_1(\mathbf{x}_{[0,t-T]}) \times \mathbb{R}^{(m+mn)T}$.
Left multiplying by $\left[ \mathbf{A} \ \mathbf{B} \ \mathbf{0} \ \mathbf{N} \ \mathbf{0} \right]$ on both sides, we obtain $\mathbf{A} \operatorname{Im} \mathcal{H}_1(\mathbf{x}_{[0,t-T]}) + \operatorname{Im} \mathbf{B} + \operatorname{Im} \mathbf{N} = \operatorname{Im} \mathcal{H}_1(\mathbf{x}_{[1,t-T+1]})$.
Since $\mathbf{x}(t-T+1) \in \operatorname{Im} \mathcal{H}_1(\mathbf{x}_{[0,t-T]})$, then $\mathbf{A} \operatorname{Im} \mathcal{H}_1(\mathbf{x}_{[0,t-T]}) + \operatorname{Im} \mathbf{B} + \operatorname{Im} \mathbf{N} = \operatorname{Im} \mathcal{H}_1(\mathbf{x}_{[0,t-T]})$.
This implies $\operatorname{Im} \mathcal{H}_1(\mathbf{x}_{[0,t-T]})$ is an $\mathbf{A}$-invariant subspace containing $\operatorname{Im} \left[ \mathbf{B} \ \mathbf{N} \right]$. Since the reachable subspace of the pair $(\mathbf{A},\left[ \mathbf{B} \ \mathbf{N} \right])$ is $\mathbb{R}^n$ by Assumption~\ref{ass:controllable-pair}, and the fact that it is also the smallest $\mathbf{A}$-invariant subspace containing $\operatorname{Im} \left[ \mathbf{B} \ \mathbf{N} \right]$,  we deduce that $\mathbb{R}^n \subseteq \operatorname{Im} \mathcal{H}_1(\mathbf{x}_{[0,t-T]})$.
Therefore $\mathbb{R}^{(m + mn)T + n} \subseteq \operatorname{Im} \mathbf{x}_{[0,t-T]} \times \mathbb{R}^{(m + mn)T} = \operatorname{Im} \mathcal{G}_T(t)$, which contradicts the fact that $\mathcal{G}_T(t)$ is not full-row rank.
\end{proof}

Based on Propositions~\ref{pro:scaled} and~\ref{pro:left-kernel}, now
we introduce the online control experiment procedure in
Algorithm~\ref{alg:experiments} to ensure the data is
$T$-persistently exciting.
The underlying idea of the strategy is to increase the row rank of $\mathcal{G}_T(t)$ at each step.

\begin{algorithm}[ht]
  \caption{Online control experiment design}\label{alg:experiments}
  \begin{algorithmic}[1]
  \State \algorithmicrequire ~$\mathbf{x}(0) = \mathbf{0}$, $\|\mathbf{u}(i)\| < \epsilon$ for $i = 0,\dots,T-1$ s.t. $\mathcal{G}_T(T) \neq \mathbf{0}$, $\epsilon$ sufficient close to $0$, $t := T$, $k := 1$
  \Repeat
  \While{$\mathcal{H}_{n+k}(\mathbf{u}_{[0,t-1]})$ is full-row rank}
  \State $k \gets k+1$
  \Comment{Increase order}
  \EndWhile
  \If{\eqref{eq:verify} holds}
  \State select $\xi \in \mathbb{R}^n$, 
    $\eta = \left[ \eta_1^\top \dots \eta_T^\top \right]^\top\in \mathbb{R}^{mT}$, and 
    $\chi = \left[ \chi_1^\top \dots \chi_T^\top \right]^\top \in \mathbb{R}^{mnT}$ s.t.~\eqref{leftkernel} holds, with $\left[\eta_T^\top \ \chi_T^\top\right] \neq \mathbf{0}$
  \If{ $\eta_T^\top + \chi_{T}^\top(\mathbf{x}(t) \otimes I) \neq \mathbf{0}$}
  \State choose $\|\mathbf{u}(t)\| < \epsilon$ s.t. $\xi^\top \mathbf{x}(t - T + 1) + \eta^\top \mathbf{u}_{[t - T + 1,t]} + \chi^\top \mathbf{x \otimes u}_{[t - T + 1,t]} \neq  \mathbf{0}$ holds
  \Else 
  \State choose $\|\mathbf{u}(t)\| < \epsilon$ s.t. $\rowrank (\mathcal{H}_{n + k}( \mathbf{u}_{[0,t]} ))$ increases
  \EndIf
  \Else
  \State choose $\|\mathbf{u}(t)\| < \epsilon$ arbitrarily
  \EndIf
 \State $t \gets t+1$
 \Comment{Update iteration}
 \Until $\mathcal{G}_T(t)$ is full-row rank
 \State $L \gets t$
 \State \algorithmicensure ~ Full-row rank $\mathcal{G}_T(L)$
\end{algorithmic}
\end{algorithm}

\begin{theorem}\longthmtitle{Online control experiment design for $T$-persistently exciting data}\label{thm:experiments-design}
Let $(\mathbf{A},\mathbf{B})$ be controllable and design the control experiment for system~\eqref{system2} according to Algorithm~\ref{alg:experiments}. Then the output $\mathcal{G}_T(L)$ is full-row rank.
\end{theorem}
\begin{proof}
  Given $t \ge T$, assume $\mathcal{G}_T(t)$ is not full-row rank. If
  \eqref{eq:verify} does not hold, it is easy to see that any choice
  of $\mathbf{u}(t)$ leads to
  $\rowrank (\mathcal{G}_T(t+1)) > \rowrank
  (\mathcal{G}_T(t))$. Hence, we concentrate on the case when
  \eqref{eq:verify} holds. In this case, from
  Proposition~\ref{pro:left-kernel}, we know there exist
  $\xi \in \mathbb{R}^n$, $\eta_1, \dots, \eta_T \in \mathbb{R}^m$,
  and $\chi_1, \dots, \chi_T \in \mathbb{R}^{mn}$, with at least one
  in $\{\eta_T,\chi_T\}$ not equal to $\mathbf{0}$
  making~\eqref{leftkernel} hold. We aim to design $\mathbf{u}(t)$ to
  satisfy
  $\xi^\top \mathbf{x}(t - T + 1) + \eta^\top \mathbf{u}_{[t - T +
    1,t]} + \chi^\top \mathbf{x \otimes u}_{[t - T + 1,t]} \neq
  \mathbf{0}$ so that
  $[\mathbf{x}(t - T + 1);\mathbf{u}_{[t - T + 1,t]};\mathbf{x \otimes
    u}_{[t - T + 1,t]}]$ does not belong to
  $\operatorname{Im} \mathcal{G}_T(t)$, which ensures
  $\rowrank (\mathcal{G}_T(t+1)) > \rowrank (\mathcal{G}_T(t))$.  Such
  $u(t)$ can be found as long as
  $\eta_T^\top + \chi_{T}^\top(\mathbf{x}(t) \otimes I) \neq
  \mathbf{0}$. If this is not the case, any selection of
  $\mathbf{u}(t)$ will not affect whether the row rank of
  $\mathcal{G}_T(t)$ will increase or not at this time step. We prove
  by contradiction that this situation will not occur indefinitely
  under Algorithm~\ref{alg:experiments}. Suppose
  $\eta_T^\top + \chi_{T}^\top(\mathbf{x}(\ell) \otimes I) =
  \mathbf{0}$ holds for all $\ell \geq t$, it then follows that
  $\eta_T^\top + \chi_{T}^\top(\mathcal{H}_1(\mathbf{x}_{[t,\ell]})
  \alpha \otimes I) = \mathbf{0}$ for any
  $\alpha \in \mathbb{R}^{\ell-t+1}$ with
  $\mathbf{1}^\top \alpha = 1$. According to
  Algorithm~\ref{alg:experiments}, the order~$k$ is increased,
  followed by an input selection that makes
  $\mathcal{H}_{n+k}(\mathbf{u}_{[0,\ell]})$ full-row rank.  Let
  $\ell$ sufficiently large so that $k > t$. In this case,
  $\mathcal{H}_{n+t}(\mathbf{u}_{[0,\ell]})$ is full-row rank and,
  using Proposition~\ref{pro:scaled},
  $\mathcal{H}_t(\mathbf{x}_{[1,\ell]})$ is full-row rank too. The
  latter implies that $\mathcal{H}_1(\mathbf{x}_{[t,\ell]})$ is
  full-row rank. Since at least one in
  $\{\eta_T,\chi_T\}$ is not equal to $\mathbf{0}$, there must exist
  $\alpha \in \mathbb{R}^{\ell-t+1}$ with $\mathbf{1}^\top \alpha = 1$
  such that
  $\eta_T^\top + \chi_{T}^\top(\mathcal{H}_1(\mathbf{x}_{[t,\ell]})
  \alpha \otimes I) \neq \mathbf{0}$ holds, which is a
  contradiction. This shows that
  Algorithm~\ref{alg:experiments} increases the row rank of
  $\mathcal{G}_T(t)$ by one after finitely many steps, and hence, it
  eventually terminates with a full-row rank~$\mathcal{G}_T(L)$.
\end{proof}

Note that the controllability assumption on the pair
$(\mathbf{A},\mathbf{B})$ is only necessary to ensure
Algorithm~\ref{alg:experiments} is successful,
cf. Theorem~\ref{thm:experiments-design}. Our design methodology below
is still valid as long as a full row-rank matrix $\mathcal{G}_T(L)$
can be obtained.

\section{Data-driven control design}\label{sec:data-driven-design}

Here, we describe an algorithmic procedure to find a local solution of
the optimal control problem~\eqref{problem} using
$T$-persistently exciting data. Our first step is to
provide an equivalent data-based representation of the
optimization. We then iteratively apply a convex-concave procedure to
solve it efficiently. The next result provides a data-based
formulation of~\eqref{problem}, provided the available data is
$T$-persistently exciting.

\begin{theorem}\longthmtitle{Data-based reformulation of optimal control problem}\label{thm:equivalent}
  Given a $T$-persistently exciting data set
  $\mathbf{x}_{\{0,L\}}$ and $\mathbf{u}_{\{0,L-1\}}$,~\eqref{problem}
  is equivalent to the data-based optimization:
\begin{align}\label{new-problem}
   \min_{\alpha} \quad & \sum_{t=0}^{T-1} \bar{\mathbf{x}}^\top(t) \mathbf{Q} \bar{\mathbf{x}}(t) + \bar{\mathbf{u}}^\top(t) \mathbf{R} \bar{\mathbf{u}}(t) \notag \\
   {\rm s.t.} \ &
    \left[ \begin{array}{c} \bar{\mathbf{x}}_{[0, T]} \\ \hline \bar{\mathbf{u}}_{[0, T-1]} \end{array} \right] = \left[ \begin{array}{c} \mathcal{H}_{T+1}(\mathbf{x}_{[0,L]}) \\ \hline \mathcal{H}_{T}(\mathbf{u}_{[0, L-1]}) \end{array} \right] \alpha,
   \notag
   \\
& \ \bar{\mathbf{x}}(0) = \mathbf{x_0}, \bar{\mathbf{x}}(T) = \mathbf{x_f}, {\rm \eqref{constraint-alpha} \ holds}.
   \tag{P2}
\end{align}
\end{theorem}

The proof of this result readily follows from Lemma~\ref{lemma:data-based-presentation}.
Notice that the optimization problem~\eqref{new-problem} is nonconvex because of the presence of bilinear terms $\alpha_i \alpha_j$, $i,j \in \{1,\dots,L-T+1\}$, in the constraints. Here, we describe a convex–concave procedure from~\cite{TL-SB:16} that can be iteratively employed to solve it.
We first describe the bilinear terms with new variables $r_{i,j} = \alpha_i \alpha_j$, which we employ in the constraints in \eqref{new-problem} to make them all become affine. We represent this set of constraints by $\mathcal{A}_1(\alpha,r)=\mathbf{0}$. Additionally, we write each equality $r_{i,j} = \alpha_i \alpha_j$ with the following equivalent representation
\begin{align*}
& (\alpha_i + \alpha_j)^2 - (\alpha_i^2 + \alpha_j^2) - 2 r_{i,j} \leq 0 ,
 \\
& (\alpha_i^2 + \alpha_j^2) - (\alpha_i + \alpha_j)^2 + 2 r_{i,j} \leq 0 .   
\end{align*}
We gather all these new nonconvex constraints in the expression
$\mathcal{C}_1(\alpha) - \mathcal{C}_2(\alpha) + \mathcal{A}_2(r) \leq
0$, where $\{\mathcal{C}_i\}_{i=1}^2$, and $\mathcal{A}_2$ are
vector-valued convex and affine functions, resp.  Using
$\mathcal{C}_0$ to denote the convex cost
function,~\eqref{new-problem} reads
\begin{align}\label{new-problem-convex}
   \min_{\alpha, r} \quad & \mathcal{C}_0(\alpha) \notag \\
   {\rm s.t.} \quad & \mathcal{C}_1(\alpha) - \mathcal{C}_2(\alpha) + \mathcal{A}_2(r) \leq \mathbf{0},  \tag{P3} \\  
   & \mathcal{A}_1(\alpha,r) = \mathbf{0}.  \notag
\end{align}
The inequality in~\eqref{new-problem-convex} can be convexified
linearizing the concave function $- \mathcal{C}_2$. We perform such
convexification iteratively to yield Algorithm~\ref{alg:ccp},
which has non-polynomial time
  complexity~\cite{TL-SB:16}. At each iteration, the algorithm solves
  a convex quadratically constrained quadratic program, with
  complexity~\cite{SM-JS:91} $O(\sqrt{M}(M+N)N^2)$ ($M$ constraints
  and~$N$ variables). Here, $M=(L-T)^2+L+T+3$ and
  $N=\frac{(L-T+1)(L-T+2)}{2}$. The next result follows
from~\cite[Sec.~1.3]{TL-SB:16}.

\begin{lemma}\longthmtitle{Convergence to critical point of~\eqref{new-problem}}
  Given a feasible initial point $\alpha^0$, all iterates of Algorithm
  \ref{alg:ccp} are feasible,
  $\{\mathcal{C}_0(\alpha^{k})\}_{k=1}^\infty$ decreases
  monotonically, and $\{\alpha^{k}\}_{k=1}^\infty$ converges to a
  critical point $\alpha^\star$ of~\eqref{new-problem}.
\end{lemma}

\begin{algorithm}[ht]
  \caption{Convex-concave procedure to solve (\eqref{new-problem-convex})}\label{alg:ccp}
  \begin{algorithmic}[1]
  \State \algorithmicgiven ~Initial feasible point $\alpha^0$, $k := 0$.
  \Repeat 
  \State Let $\bar{\mathcal{C}}_2(\alpha,\alpha^k) \triangleq \mathcal{C}_2(\alpha^k) + \nabla_\alpha \mathcal{C}_2(\alpha^k)^\top (\alpha - \alpha^k)$
  \Comment{Convexifying the constraint}
  \State Set $\alpha^{k+1}$ to be the solution of the convex problem
  \begin{align*}
   \min_{\alpha, r} \quad & \mathcal{C}_0(\alpha) \notag \\
   {\rm s.t.} \quad & \mathcal{C}_1(\alpha) - \bar{\mathcal{C}}_2(\alpha,\alpha^k) + \mathcal{A}_2(r) \leq 0  \\  
   & \mathcal{A}_1(\alpha,r) = 0
\end{align*}
 \State $k \gets k+1$
 \Comment{Update iteration}
 \Until convergence
  \end{algorithmic}
\end{algorithm}

\vspace*{-2ex}
\section{Simulation examples}\label{sec:4}

Here we illustrate the effectiveness of the proposed data-based approach in solving~\eqref{problem} and compare it against the model-based approaches in~\cite{SW-JL:17,YZ-JC:17-tcns} for bilinear systems. 

\begin{figure*}[ht]
    \centering
        \subfigure[Example \ref{example1}]{
        \centering
        \includegraphics[width=.315\linewidth]{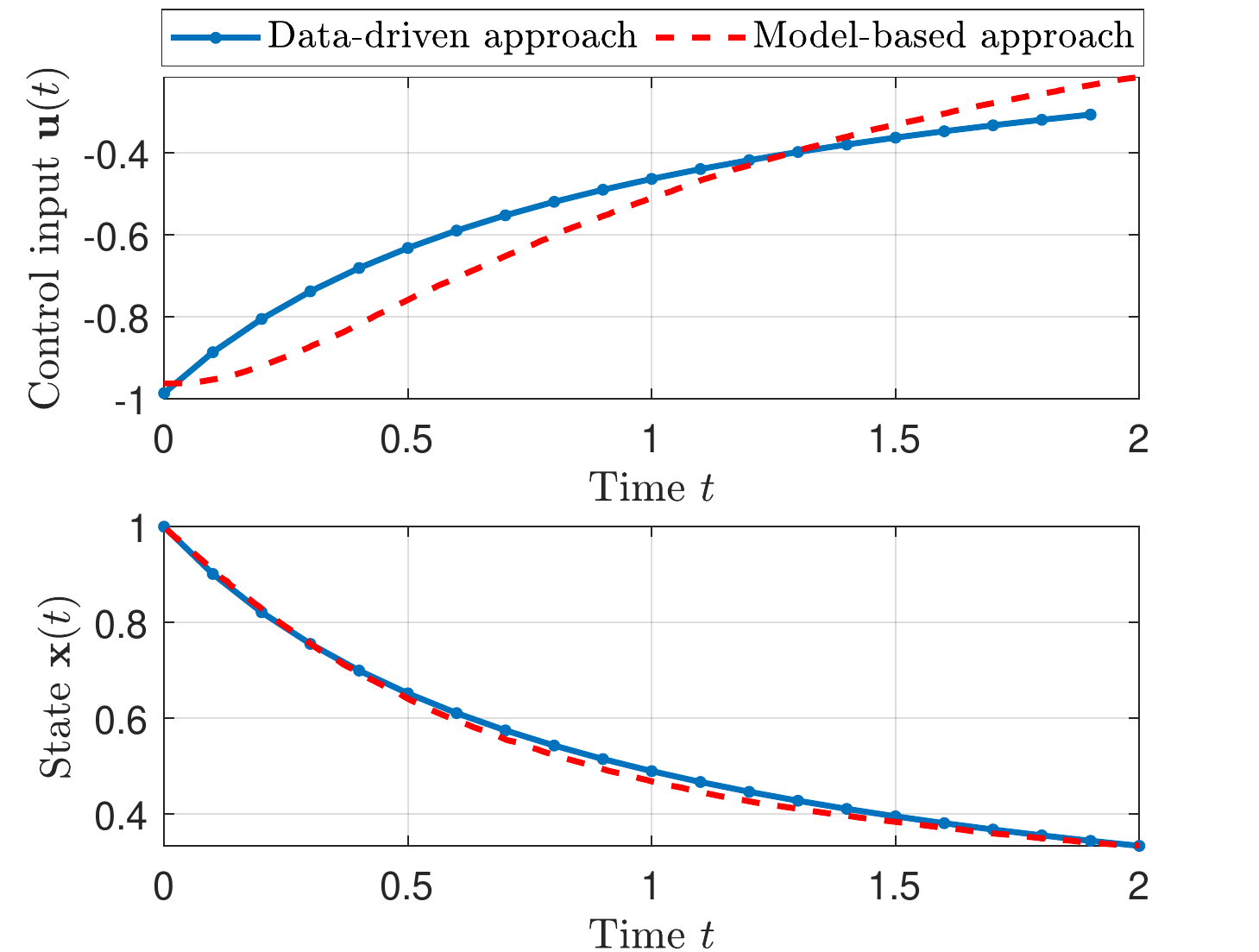}
     \label{fig1}}
    \subfigure[Example \ref{example2}]{
        \centering
        \includegraphics[width=.315\linewidth]{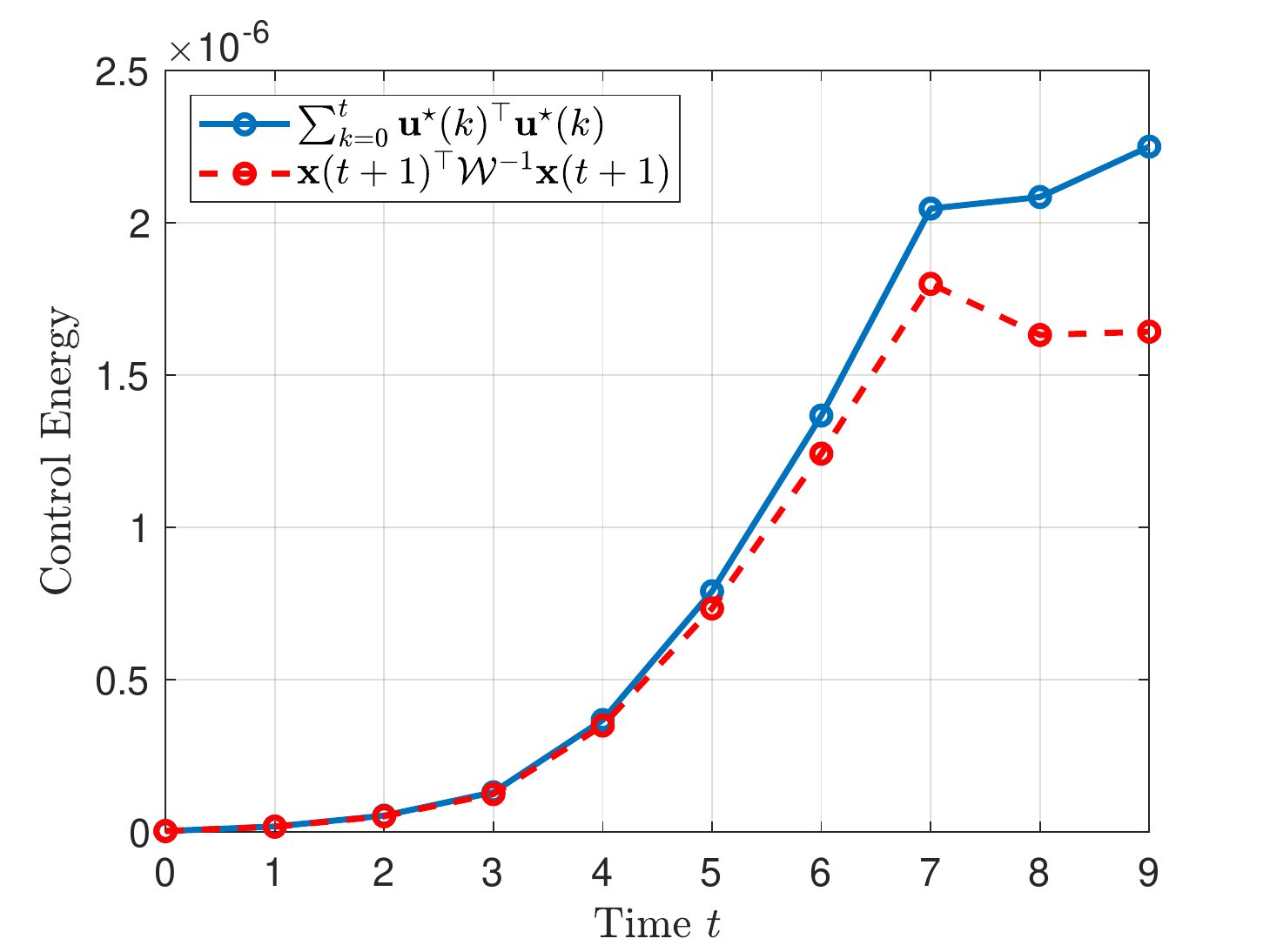}
       \label{fig2}}
           \subfigure[Example \ref{example3}]{
        \centering
        \includegraphics[width=.315\linewidth]{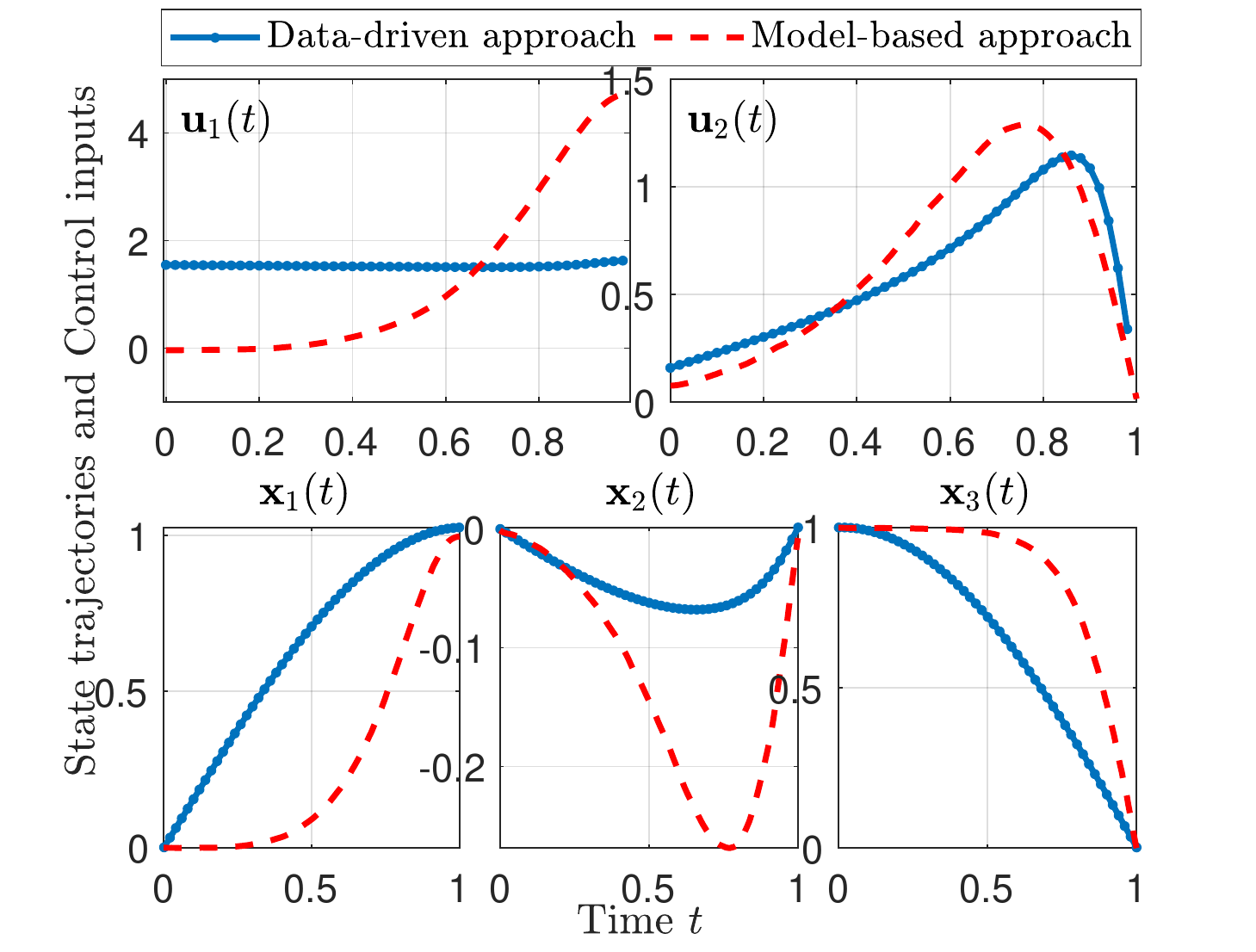}
        \label{fig3}}
\caption{Performance of the proposed data-driven approach (solid blue lines) versus model-based approaches (dashed red lines). 
The total cost values are (a) $0.1 \times \sum_{t = 0}^{19} \mathbf{x}^2(t) + \mathbf{u}^2(t) = 1.3346$ for the data-based approach and $ \int_{0}^{2} \mathbf{x}^2(\tau) + \mathbf{u}^2(\tau) d\tau = 1.3506$ for the model-based iterative method in~\cite{SW-JL:17}; (b) $\sum_{t=0}^{9}{\mathbf{u}(t)}^\top \mathbf{u}(t)= 2.25 \times 10^{-6}$ for the data-based approach and $\mathbf{x}(10)^\top\mathcal{W}^{-1}\mathbf{x}(10) = 1.64 \times 10^{-6}$ for the Gramian-based lower bound in~\cite{YZ-JC:17-tcns}; and (c) $0.02 \times \sum_{t = 0}^{49} \mathbf{u}^\top(t)\mathbf{u}(t) = 2.7999$ for the data-based approach and $\int_{0}^{1} \mathbf{u}^\top(\tau)\mathbf{u}(\tau) d\tau = 4.7976$ for the model-based iterative method in~\cite{SW-JL:17}.}
\vspace*{-2.5ex}
\end{figure*}

\begin{example}\longthmtitle{Population control}\label{example1}
  We consider a population control problem introduced in~\cite[Example
  1]{SW-JL:17} evolving in continuous time. For the horizon $T = 20$,
  we use a first-order Euler discretization with stepsize $0.1$. The
  resulting discrete-time bilinear system is
  $\mathbf{x}(t+1) = \mathbf{x}(t) + 0.1\mathbf{x}(t)\mathbf{u}(t)$.
  We take $\mathbf{Q} = \mathbf{R} = 1$ and consider
  $\mathbf{x_0} = 1$, $\mathbf{x_f} = \frac{1}{3}$. We perform a
  control experiment with $L = 60$ using randomly generated inputs,
  and verify that the resulting $\mathcal{G}_{20}(60)$ is full-row
  rank. Algorithm~\ref{alg:ccp} obtains a local optimum $\alpha^\star$
  of~\eqref{new-problem}. Fig.~\ref{fig1} shows the trajectories, both
  displaying similar performance, obtained from the data-based
  solution in Theorem~\ref{thm:equivalent} with that of the
  model-based iterative method~\cite{SW-JL:17}.  \hfill $\rhd$
\end{example}

\begin{example}\longthmtitle{Minimum-energy control problem}\label{example2}
  Consider the bilinear system from~\cite[Example 4.5]{YZ-JC:17-tcns},
  $ \mathbf{x}(t+1) = \mathbf{A} \mathbf{x}(t) + \mathbf{B}
  \mathbf{u}(t) + \mathbf{N}\mathbf{x}(t)\mathbf{u}(t)$, where
  $\mathbf{N} = \operatorname{diag}(0.1,0.2,0.3,0.4,0.5)$,
\begin{align*}
  \mathbf{A}
  = \left[
  \begin{array}{ccccc}
    0 & 0 & 0.024 & 0 & 0 \\
    1 & 0 & -0.26 & 0 & 0 \\
    0 & 1 & 0.9 & 0 & 0
    \\
    0 & 0 & 0.2 & 0 & -0.06
    \\
    0 & 0 & 0.15 & 1 & 0.5
  \end{array}
                       \right],
                       \; 
                       \mathbf{B} = \left[\begin{array}{l}
                                            0.8 \\
                                            0.6 \\
                                            0.4 \\
                                            0.2 \\
                                            0.5
                                          \end{array}\right].
\end{align*}
We consider the minimum-energy control problem
($\mathbf{Q} = \mathbf{0}$, $\mathbf{R} = I$) with $T = 10$. Let
$\mathbf{x_0} = \mathbf{0}$ and
$\mathbf{x_f} = \left[ 0.0004 \ -0.00038 \ 0.00318 \ 0.00062 \ 0.00219
\right]^\top$.  We perform a control experiment with $L = 74$ using
Algorithm~\ref{alg:experiments}. The execution of
  Algorithm~\ref{alg:experiments} here increases the row rank of
  $\mathcal{G}_T(t)$ monotonically for every $t \geq T$ until it
  becomes full-row rank (i.e., the algorithm never falls into Step
  11). We solve \eqref{new-problem} using
Algorithm~\ref{alg:ccp}. For comparison, we use the Gramian-based
lower bound of the optimal cost value obtained in\cite{YZ-JC:17-tcns},
\begin{align*}
  \sum_{t=0}^{T-1} {\mathbf{u}^\star}^\top(t) \mathbf{u}^\star(t) \geq \mathbf{x}^\top(T) \mathcal{W}^{-1} \mathbf{x}(T),  
\end{align*}
where $\mathcal{W}$ is the reachability Gramian of the bilinear system. Fig.~\ref{fig2} compares this lower bound with the values obtained with the trajectories from the data-based solution in Theorem~\ref{thm:equivalent}, showing a close agreement between the two. \hfill$\rhd$
\end{example}

\begin{example}\longthmtitle{Minimum-energy control problem}\label{example3}
  We consider a minimum-energy control example from~\cite[Example
  2]{SW-JL:17}, for which we use a first-order Euler discretization
  with stepsize $0.02$. The discrete-time bilinear system is
  $ \mathbf{x}(t+1) = \mathbf{A}\mathbf{x}(t) +
  \mathbf{B}\mathbf{u}(t) + \Big[ \sum_{j=1}^{3}\mathbf{x}_j(t)
  \mathbf{N}_j \Big] \mathbf{u}(t)$, with
\begin{align*}
	&\mathbf{A} = \left[ \begin{array}{ccc}
		1 & -0.01 & 0 \\
		0.01 & 1 & 0\\
		0 & 0 & 1
	\end{array} \right], \mathbf{B} = \mathbf{0},  \mathbf{N}_1 = \left[ \begin{array}{cc}
		0 & 0 \\
		0 & 0 \\
		-0.02 & 0
	\end{array} \right],\\
	&\mathbf{N}_2 = \left[ \begin{array}{cc}
		0 & 0 \\
		0 & 0 \\
		0 & 0.02
	\end{array} \right],  \mathbf{N}_3 = \left[ \begin{array}{cc}
		0.02 & 0 \\
		0 & -0.02 \\
		0 & 0
	\end{array} \right].
\end{align*}
We consider $T = 50$ and perform a control experiment with $L = 452$
randomly generated inputs,
and verify $\mathcal{G}_{50}(452)$ is full-row rank. We let
$\mathbf{x_0} = \left[ 0 \ 0 \ 1 \right]^\top$,
$\mathbf{x_f} = \left[ 1 \ 0 \ 0 \right]^\top$. We solve
\eqref{new-problem} using Algorithm~\ref{alg:ccp} to obtain
$\alpha^\star$ and compare, cf. Fig.~\ref{fig3}, the trajectories
obtained from the data-based solution in Theorem~\ref{thm:equivalent}
with that of the model-based iterative method~\cite{SW-JL:17}, showing
a better local optimum by the former.~$\rhd$
\end{example}

\vspace*{-2ex}
\section{Conclusions}\label{sec:conclusions}
We have presented a data-driven method to learn optimal controls of
bilinear systems directly from input/state data without a priori
knowledge of the matrices. We have provided an online control
experiment design method to obtain sufficiently informative data and
introduced an equivalent data-based reformulation of the original
nonconvex optimal control problem and employed an iterative
convex-concave algorithmic procedure to solve it. Simulations show
data-based optimal control trajectories have comparable performance to
those obtained by model-based ones.  Future work will explore
extensions to noisy data and robustness analysis, weaker notions under
which data is sufficient to reconstruct optimal controls, online
implementations of the convex-concave procedure as data becomes
increasingly available, and distributed implementations for
large-scale bilinear networks.

\end{document}